\documentclass[10pt]{amsart}
\usepackage{latexsym}
\usepackage{amsfonts}

\usepackage{tikz,stmaryrd}
\usetikzlibrary{arrows,chains,matrix,positioning,scopes}
\makeatletter
\tikzset{join/.code=\tikzset{after node path={%
\ifx\tikzchainprevious\pgfutil@empty\else(\tikzchainprevious)%
edge[every join]#1(\tikzchaincurrent)\fi}}}
\makeatother
\tikzset{>=stealth',every on chain/.append style={join},
         every join/.style={->}}
\tikzstyle{labeled}=[execute at begin node=$\scriptstyle,
   execute at end node=$]

\usepackage{amsmath,amsthm,amssymb}

\author[I.~Kapovich]{Ilya Kapovich}

\address{\tt Department of Mathematics, University of Illinois at
  Urbana-Champaign, 1409 West Green Street, Urbana, IL 61801, U.S.A.
  \newline http://www.math.uiuc.edu/\~{}kapovich/} \email{\tt kapovich@math.uiuc.edu}

\author[K.~Rafi]{Kasra Rafi}
\address{\tt Department of Mathematics, University of Oklahoma,
Norman, OK 73019, U.S.A.
  \newline http://www.math.ou.edu/\~{}rafi/}
\email{\tt rafi@math.ou.edu}

\title{On hyperbolicity of free splitting and free factor complexes}

\newtheorem{thm}{Theorem}[section] 
\newtheorem{cor}[thm]{Corollary} 
\newtheorem{prop}[thm]{Proposition} \theoremstyle{definition}
\newtheorem{defn}[thm]{Definition}

\newtheorem{conv}[thm]{Convention} \newtheorem{rem}[thm]{Remark}

\def\strutdepth{\dp\strutbox}
\def \ss{\strut\vadjust{\kern-\strutdepth \sss}}
\def \sss{\vtop to \strutdepth{
\baselineskip\strutdepth\vss\llap{$\diamondsuit\;\;$}\null}}

\def\strutdepth{\dp\strutbox}
\def \sst{\strut\vadjust{\kern-\strutdepth \ssss}}
\def \ssss{\vtop to \strutdepth{
\baselineskip\strutdepth\vss\llap{$\spadesuit\;\;$}\null}}

\def\strutdepth{\dp\strutbox}
\def \ssh{\strut\vadjust{\kern-\strutdepth \sssh}}
\def \sssh{\vtop to \strutdepth{
\baselineskip\strutdepth\vss\llap{$\heartsuit\;\;$}\null}}

\newcommand{\cal}{\mathcal}

\newcommand{\CG}{{\cal G}}

\def\epsilon{\varepsilon}
\def\phi{\varphi}

\newcommand{\Out}{\mbox{Out}}

\newcommand{\from}{\colon\,}

\newcommand{\cvn}{\mbox{cv}_N}

\DeclareMathOperator{\diam}{diam}
\DeclareMathOperator{\Core}{Core}
\DeclareMathOperator{\degree}{deg}

\begin{document}

\begin{abstract}
We show how to derive hyperbolicity of the free factor complex of $F_N$ from
the Handel-Mosher proof of hyperbolicity of the free splitting complex of $F_N$, 
thus obtaining an alternative proof of a theorem of Bestvina-Feighn. We also show that under the natural map $\tau$ from the free splitting complex to free factor complex, a geodesic $[x,y]$ maps to a path that is uniformly Hausdorff-close to a geodesic $[\tau(x),\tau(y)]$ . 
\end{abstract}

\thanks{The first author was partially supported by the NSF grant DMS-0904200. The second author was partially supported by the NSF grant DMS-1007811.}


\maketitle

\section{Introduction}

The notion of a \emph{curve complex}, introduced by Harvey~\cite{Har}
in late 1970s, plays a key role in the study of hyperbolic surfaces,
mapping class group and the Teichm\"uller space.

If $S$ is a compact connected oriented surface, the \emph{curve
  complex} $\mathcal C(S)$ of $S$ is a simplicial complex whose
vertices are isotopy classes of essential non-peripheral simple closed
curves. A collection $[\alpha_0],\dots, [\alpha_n]$ of $(n+1)$ distinct
vertices of $\mathcal C(S)$  spans an $n$--simplex in $\mathcal C(S)$
if there exist representatives $\alpha_0,\dots, \alpha_n$ of these
isotopy classes such that for all $i\ne j$ the curves $\alpha_i$ and
$\alpha_j$ are disjoint. (The definition of $\mathcal C(S)$ is a little different for several surfaces of small genus). The complex $\mathcal C(S)$ is
finite-dimensional but not locally finite, and it comes equipped with
a natural action of the mapping class group $Mod(S)$ by simplicial
automorphisms. It turns out that the geometry of $\mathcal C(S)$ is
closely related to the geometry of the Teichm\"uller space $\mathcal
T(S)$ and also of the mapping class group itself. The curve complex is
a basic tool in modern Teichm\"uller theory, and has also found numerous
applications in the study of 3-manifolds and of Kleinian groups.
A key general result of Masur and Minsky~\cite{MM99} says that the curve complex $\mathcal C(S)$, equipped with
the simplicial metric, is a Gromov-hyperbolic space. Hyperbolicity of
the curve complex was an important ingredient in the solution by
Masur, Minsky, Brock and Canary  of the Ending Lamination
Conjecture~\cite{MM00,M10,BCM} (see \cite{M03} for detailed background discussion about this solution).

The outer automorphism group $\Out(F_N)$ of a free group $F_N$ is a
cousin of the mapping class group. However the
group $\Out(F_N)$ is much less well understood and, in general, more
difficult to study than the mapping class group. A free group analog of
the Teichmuller space is the Culler-Vogtmann Outer space $\cvn$,
introduced by Culler and Vogtmann in~\cite{CV}. The points of $\cvn$
are free minimal discrete isometric actions of $F_N$ on $\mathbb
R$--trees, considered up to $F_N$--equivariant isometry. The Outer space
comes equipped with a natural action of $\Out(F_N)$. It is known that
$\cvn$ is finite-dimensional and contractible; as a result, quite a
bit is known about homotopy properties of $\Out(F_N)$. 
However, the geometry of $\cvn$ and of $\Out(F_N)$ proved to be much
more difficult to tackle, particularly because $\cvn$ lacks the
various useful analytic and geometric structures present in the
Teichm\"uller space case. Another problem is that many geometric
dualities from the world of Riemann surfaces and their homeomorphisms
break down for automorphisms of free groups.

In the case of a compact connected oriented surface $S$, an essential
non-peripheral simple closed curve $\alpha$ on $S$ can be viewed in
several other ways. Thus one can view $[\alpha]$ as a conjugacy
class in the fundamental group $\pi_1(S)$. We may also think of
$\alpha$ as corresponding to the (possibly disconnected) subsurface
$K_\alpha$ of
$S$ obtained by cutting $S$ open along $\alpha$. Third, $\alpha$
determines a splitting of $\pi_1(S)$ as an amalgamated product or an
HNN-extension (depending on whether the curve $\alpha$ is separating
or non-separating) over the infinite cyclic subgroup $\langle \alpha
\rangle$. We can interpret adjacency of vertices in $\mathcal C(S)$ using
each of these points of views - or a combination of them, providing several 
essentially equivalent
descriptions of the curve complex. Thus two distinct vertices
$[\alpha], [\beta]$ of $\mathcal C(S)$ are adjacent if and only if
$\alpha$ is conjugate in $\pi_1(S)$ to an element of a vertex group of
the cyclic splitting of $\pi_1(S)$ corresponding to
$\beta$. Equivalently, $[\alpha]$ and $[\beta]$ of $\mathcal C(S)$ are
adjacent if and only if the cyclic splittings of $\pi_1(S)$
corresponding to $[\alpha]$ and $[\beta]$ admit a common refinement,
that is, a splitting of $\pi_1(S)$ as the fundamental group of a graph
of groups with two edges and cyclic edge groups, such that collapsing
one of the edges produces a splitting corresponding to  $[\alpha]$ and
collapsing the other edge produces a splitting corresponding to
$[\beta]$. Also, $[\alpha]$ and $[\beta]$ are adjacent in $\mathcal
C(S)$ if and only if there are connected components $K_\alpha'$ of
$K_\alpha$ and $K_\beta'$ of $K_\beta$ such that $K_\alpha'\subseteq
K_\beta'$ or $K_\beta'\subseteq K_\alpha'$.

In the case of $F_N$ these different points of view produce several
possible analogs of the notion of a curve complex that are no longer
essentially equivalent. 
The first of these is the \emph{free splitting complex} $FS_N$. The
vertices of $FS_N$ are nontrivial splittings of the type $F_N=\pi_1(\mathbb A)$
where $\mathbb A$ is a graph of groups with a single edge (possibly a
loop edge) and the trivial edge group; two such splittings are
considered to be the same if their Bass-Serre covering trees are
$F_N$--equivariantly isometric. Two distinct vertices $\mathbb A$ and
$\mathbb B$ of $FS_N$ are joined by an edge if these splittings admit
a common refinement, that is, a splitting $F_N=\pi_1(\mathbb D)$ where
$\mathbb D$ is a graph of groups with two edges and trivial edge
groups, such that collapsing one edge gives the splitting $\mathbb A$
and collapsing the other edge produces the splitting $\mathbb
B$. Higher-dimensional simplices are defined in a similar way, see
Definition~\ref{defn:FS_N} below for a careful formulation.
For example, if $F_N=A\ast B\ast C$, where $A,B,C$ are nontrivial, then the
splittings $F_N=(A\ast B)\ast C$ and $F_N=A\ast (B\ast C)$ are
adjacent in $FS_N$. There is a natural action of $\Out(F_N)$ on $FS_N$
by simplicial automorphisms. 
The above definition of $FS_N$ has a variation~\cite{SabSav}, called the
\emph{edge-splitting complex}, denoted $ES_N$, where in the definition
of vertices only splittings $\mathbb A$ with a single \emph{non-loop}
edge are allowed.

A rather different free group analog of the curve complex is the
\emph{free factor complex} $FF_N$, originally introduced by Hatcher
and Vogtmann~\cite{HV}. The vertices of $FF_N$ are conjugacy classes
$[A]$ of proper free factors $A$ of $F_N$. Two distinct vertices $[A],
[B]$ are joined by an edge in $FF_N$ if there exist representatives
$A$ of $[A]$ and $B$ of $[B]$ such that $A\le B$ or $B\le
A$. Higher-dimensional simplices are defined similarly,
see~Definition~\ref{defn:FF_N} below. Note that this definition does
not work well for $N=2$ as it produces a graph consisting of isolated
vertices corresponding to conjugacy classes of primitive elements in
$F_2$. However, there is a natural modification of the definition of
$FF_N$ for $N=2$ (see~\cite{BF11}) such that $FF_2$ becomes the
standard Farey graph (and in particular $FF_2$ is hyperbolic).

A closely related object to $FF_N$ is the \emph{simplicial intersection graph}
$I_N$. The graph $I_N$ is a bipartite graph with two types of
vertices: single-edge free splittings $F_N=\pi_1(\mathbb A)$ (that is, vertices of
$FS_N$) and conjugacy classes of \emph{simple} elements of $F_N$. Here
an element $a\in F_N$ is \emph{simple} if $a$ belongs to some proper
free factor of $F_N$. A free splitting $\mathbb A$ and a conjugacy
class $[a]$ of a simple element $a$ are adjacent if $a$ is conjugate
to an element of a vertex group of $\mathbb A$. The graph is a subgraph of a more general ``intersection graph" defined in \cite{KL09}.

Both $FF_N$ and $I_N$ admit natural $\Out(F_N)$--actions. It is also not hard to check that for $N\ge 3$
the graph $I_N$ is quasi-isometric to the free factor complex
$FF_N$. By contrast, the free factor complex $FF_N$ and the free
splitting complex $FS_N$ are rather different objects
geometrically. By construction, the vertex set $V(FS_N)$ is a $1$--dense subset of
$V(I_N)$. Also, the inclusion map $\iota\from (V(FS_N), d_{FS_N})\to (I_N,
d_{I_N})$ is $2$--Lipschitz. However  the distance between two free splittings in $I_N$ is generally much smaller
than the distance between them in $FS_N$. Intuitively, it is ``much
easier'' for $\mathbb A$ and $\mathbb B$ to share a common elliptic
simple element (which would make $d_{I_N}(\mathbb A, \mathbb B)\le 2$)
then for these splittings to admit a common refinement.

Until recently, basically nothing was known about the geometry of the
above complexes. Several years ago Kapovich-Lustig~\cite{KL09} and
Behrstock-Bestvina-Clay~\cite{BBC} showed that for $N\ge 3$ the
(quasi-isometric) complexes $FF_N$ and $I_N$ have infinite diameter.
Since the inclusion map $\iota$ above is Lipschitz, this implies that
$FS_N$ has infinite diameter as well. A subsequent result of
Bestvina-Feighn~\cite{BF10} implies that every fully irreducible
element $\phi\in \Out(F_N)$ acts on $FF_N$ with positive asymptotic
translation length (hence the same is true for the action of $\phi$ on
$FS_N$). It is easy to see from the definitions that if  $\phi\in
\Out(F_N)$ is not fully irreducible then some positive power of $\phi$ fixes a
vertex of $FF_N$, so that $\phi$ acts on $FF_N$ with bounded orbits.

Sabalka and Savchuk proved~\cite{SabSav} in 2010 that the edge-splitting complex
$ES_N$ is not Gromov-hyperbolic, because it possesses some
quasi-flats. Aramayona and Souto~\cite{AS} showed that every automorphism of $FS_N$ is induced by some element of $\Out(F_N)$.

Last year (2011), two significant further advances occurred. First,
Bestvina and Feighn~\cite{BF11} proved that for $N\ge 2$ the free
splitting complex is Gromov-hyperbolic (as noted above, for $N=2$ this
essentially follows from the definition of $FF_2$, so the main case of
the Bestvina-Feighn result is for $N\ge 3$). Then Handel and
Mosher~\cite{HM} proved that for all $N\ge 2$ the free splitting complex
$FS_N$ is also Gromov-hyperbolic. The two proofs are rather different
in nature, although both are quite complicated. Recently Hilion and Horbez~\cite{HH} produced another proof of hyperbolicity of $FS_N$, using ``surgery paths" in the sphere complex model of $FS_N$.  Bestvina and Reynolds~\cite{BR12} and Hamenstandt~\cite{Ha12} gave a description of the hyperbolic boundary of $FF_N$. Also, Bestvina and Feighn~\cite{BF12} and then Sabalka and Savchuk~\cite{SS12} investigated analogs of subsurface projections in the $FS_N$ and $FF_N$ contexs.

In the present paper we show how to derive hyperbolicity of the free
factor complex from the Handel-Mosher proof of hyperbolicity of the
free splitting complex. This gives a new proof of the
Bestvina-Feighn result~\cite{BF11} about hyperbolicity of $FF_N$.

There is a natural ``almost canonical'' Lipschitz projection from the free splitting complex to a free factor complex.
Namely, for any free splitting $v=\mathbb A\in V(FS_N)$ choose a vertex $u$ of $\mathbb A$ and put $\tau(v)\from=[A_u]$, where $A_u$ is the vertex group of $u$ in $\mathbb A$. This defines a map (easily seen to be Lipschitz) $\tau\from V(FS_N)\to V(FF_N)$. Extend this map to a graph-map $\tau\from FS_N^{(1)}\to FF_N^{(1)}$
by sending every edge in $FS_N$ to a geodesic in $FF_N^{(1)}$ joining the $\tau$--images of the endpoints of that edge.
Although the map $\tau$ is not quite canonically defined (since it involves choosing a vertex group in a free splitting $\mathbb A$ when defining $\tau(\mathbb A)$), it is easy to check that, for $N\ge 3$, if $\tau'\from  V(FS_N)\to V(FF_N)$ is another map constructed by the above procedure, then $d(\tau(v), \tau'(v))\le 2$ for all $v\in V(FS_N)$.  

We prove:

\begin{thm}\label{thm:A}
Let $N\ge 3$. Then the free factor complex $FF_N$ is Gromov-hyperbolic. Moreover, there exists a constant $C>0$ such that for any two vertices $x,y$ of $FS_N$ and any geodesic $[x,y]$ in $FS_N^{(1)}$ the path $\tau([x,y])$ is $C$--Hausdorff close to a geodesic $[\tau(x),\tau(y)]$ in $FF_N^{(1)}$.
\end{thm}

To prove Theorem~\ref{thm:A}, we first introduce a new object, called the \emph{free bases graph},
and denoted $FB_N$, see Definition~\ref{defn:FB_N} below. The vertices
of $FB_N$ are free bases of $F_N$, up to some natural
equivalence. Informally, adjacency in $FB_N$ corresponds to two free
bases sharing a common element. We then prove
(Proposition~\ref{prop:FB}) that the natural map from $FB_N$ to $FF_N$ is a quasi-isometry.
Thus to show that $FF_N$ is hyperbolic
it suffices to establish hyperbolicity of $FB_N$.
To do the latter we use a hyperbolicity criterion for graphs (Proposition~\ref{prop:bow} below) due to
Bowditch~\cite{B}. Roughly, this criterion requires that there exists a family of paths $\mathcal G=\{g_{x,y}\}_{x,y}$ (where $x,y\in VX$) joining
$x$ to $y$ and that there exists a ``center"-like map $\Phi\from V(X)\times V(X) \times V(X)\to V(X)$, such that the pair $(\mathcal G, \Phi)$ satisfies some nice ``thin triangle"  type properties, see Definition~\ref{defn:thin} below.  Using Bowditch's criterion we obtain Corollary~\ref{cor:contraction} saying that if $X,Y$ are connected graphs, with $X$ hyperbolic and if $f\from X\to Y$ is a surjective Lipschitz graph-map with the property that if $d(f(x),f(y))$ is small then $f ([x,y])$ has bounded diameter, then $Y$ is also hyperbolic. Moreover, in this case $f([x,y])$ is uniformly Hausdorff-close to any geodesic $[f(x), f(y)]$ in $Y$. (See also Theorem 3.11 in \cite{MS}.)

We then construct a surjective Lipschitz map $f\from  FS_N'\to FB_N$, where
$FS_N'$ is the barycentric subdivision of $FS_N$. The map $f$ restricts to a natural bijection from a subset $S$ of $V(FS_N')$,
corresponding to $N$--roses,  to the set $V(FB_N)$ of vertices of $FB_N$. Thus we may, by abuse of notation, say that $S=V(FB_N)$ and that $f|_S=Id_S$.
In~\cite{HM} Handel and Mosher constructed nice paths $g_{x,y}$ given by ``folding
sequences''  between arbitrary vertices $x$ and $y$ of $FS_N'$, and proved that
these paths are quasigeodesics in $FS_N'$. To apply Corollary~\ref{cor:contraction} to the map $f\from  FS_N'\to FB_N$ it turns out to be enough to show that $f(g_{x,y})$ has bounded diameter if $x,y\in S$ and $d(f(x), f(y))\le 1$ in $FB_N$.
To do that we analyze the properties of the
Handel-Mosher folding sequences in this specific situation. The 
construction of $g_{x,y}$
for arbitrary $x,y\in V(FS_N')$ is fairly complicated. However, in our
situation, we have $x,y\in S$, so that $x,y$ correspond to free bases of $F_N$. In
this case the construction of $g_{x,y}$ becomes much easier and boils
down to using standard Stallings foldings (in the sense
of~\cite{KM,St}) to get from $x$ to $y$. Verifying that $f(g_{x,y})$
has bounded diameter in $FB_N$, assuming $d(f(x),f(y))\le 1$, becomes a
much simpler task. Thus we are able to conclude that $FB_N$ is Gromov-hyperbolic, and, moreover, that $f([x,y])$ is uniformly Hausdorff-close to any geodesic $[f(x), f(y)]$ in $FB_N$. Using the quasi-isometry between $FB_N$ and $FF_N$ provided by Proposition~\ref{prop:FB}, we then obtain the conclusion of Theorem~\ref{thm:A}.

Moreover, as we note in Remark~\ref{rem:geod}, our proof of Theorem~\ref{thm:A} provides a fairly explicit description of quasigeodesics joining arbitrary vertices (i.e. free bases) in $FB_N$ in terms of Stallings foldings.
 
{\it Acknowledgement:}  This paper grew out of our discussions during the November 2011 AIM SQuaRE Workshop ``Subgroups of $Out(F_N)$ and their action on outer space" at the American Institute of Mathematics in Palo Alto.
We thank AIM and the other participants of the workshop,
Thierry Coulbois, Matt Clay, Arnaud Hilion, Martin Lustig and Alexandra Pettet,
for the stimulating research environment at the workshop. We are also grateful to Patrick Reynolds for the many helpful comments on the initial version of this paper and to Saul Schleimer for a helpful discussion regarding references for Proposition~\ref{prop:qc}.
We especially thank the referee for the numerous useful comments and suggestions, and for providing Figure~\ref{Fi:1} (see Section~6 below) illustrating the relationships and the maps between the main objects considered in this paper.

\section{Hyperbolicity criteria for graphs}

\begin{conv}
From now on, unless specified otherwise, every connected graph $X$ will be 
considered as a geodesic metric space with the simplicial metric (where every 
edge has length $1$). As in the introduction, we denote
the vertex set of $X$ by $V(X)$. Also, when talking about a connected simplicial 
complex $Z$ as a metric space, we will in fact mean the $1$--skeleton 
$Z^{(1)}$ of $Z$ endowed with the simplicial metric.
\end{conv}

Let $X, Y$ be connected graphs. A \emph{graph-map} from $X$ to $Y$ is
a continuous function $f\from X\to Y$ such that $f(V(X))\subseteq
V(Y)$ (so that $f$ takes vertices to vertices), and such that for
every edge $e$ of $X$ $f(e)$ is an edge-path in $Y$ (where we allow
for an edge-path to be degenerate and to consist of a single vertex).
Note that if $f\from X\to Y$ is a graph-map and $X'$ is a subgraph of $X$
then $f\big|_{X'}: X'\to Y$ is also a graph-map and $f(X')$ is a subgraph of $Y$.

We say that a graph-map $f \from X\to Y$ is \emph{$L$--Lipschitz} (where
$L\ge 0$) if for every edge $e$ of $X$ the edge-path $f(e)$ has
simplicial length $\le L$.

\medskip

We use a characterization of hyperbolicity for a geodesic metric 
space $(X, d_X)$ that is due to Bowditch \cite{B}.
A similar hyperbolicity conditions have been originally stated 
by Masur and Minsky (see Theorem~2.3 in \cite{MM99}).
A related statement was also obtained by Hamenstadt~\cite{Ha}.
The following result is a slightly restated special case of
Proposition~3.1 in \cite{B}. 

\begin{defn}[Thin triangles structure]\label{defn:thin}
Let $X$ be a connected graph.
Let $\CG=\{g_{x,y}| x,y\in V(X)\}$ be a family of edge-paths in $X$ such that for any vertices $x,y$ of $X$ $g_{x,y}$ is a path from $x$ to $y$ in $X$.  Let 
$\Phi \from V(X) \times V(X) \times V(X) \to V(X)$ be a function such that for any $a,b,c \in V(X)$, 
\[
\Phi(a,b,c) = \Phi(b,c,a) =\Phi(c,a,b).
\]
Assume, for constants $B_1$ and $B_2$ that $\CG$ and $\Phi$ have the 
following properties:
\begin{enumerate}
\item For $x,y \in V(X) $, the Hausdorff distance between $g_{x,y}$ 
and $g_{y,x}$ is at most $B_2$. 
\item For $x,y \in V(X)$, $g_{x,y} \from [0,l] \to X$, $s,t \in [0,l]$
and $a,b \in V(X)$, assume that 
\[
d_X(a, g(s)) \le B_1
\qquad\text{and}\qquad
d_X(b, g(t)) \le B_1.
\]
Then, the Hausdorff distance between $g_{a,b}$ and $g_{x,y} \big|_{[s,t]}$
is at most $B_2$. 
\item For any $a,b,c \in V(X) $, the vertex $\Phi(a,b,c)$ is contained in 
a $B_2$--neighborhood of $g_{a,b}$.
\end{enumerate}
Then, we say that the pair $(\mathcal G, \Phi)$ is a 
\emph{$(B_1, B_2)$--thin triangles structure} on $X$.
\end{defn}


\begin{prop}[Bowditch] \label{prop:bow}
Let $X$ be a connected graph. For every $B_1>0$ and $B_2>0$, there is
$\delta>0$ and $H>0$ so that if $(\CG, \Phi)$ is a $(B_1,B_2)$--thin 
triangles structure on $X$ then $X$ is $\delta$--hyperbolic. Moreover, every path 
$g_{x,y}$ in $\mathcal G$ is  $H$--Hausdorff-close to any geodesic 
segment $[x,y]$.
\end{prop}

\begin{cor} \label{cor:contraction} 
Let $X$ and $Y$ be connected graphs and assume that $X$ is 
$\delta_0$--Gromov-hyperbolic.  Let $f\from X\to Y$ be a $L$--Lipschitz 
graph-map and  $f(V(X))=V(Y)$. Suppose there are integers $M_1>0$ and 
$M_2>0$ so that, for $x, y \in V(X)$, if $d_Y(f(x),f(y))\le M_1$ then 
$\diam_Y f([x,y]) \le M_2$.

Then, there exists $\delta_1>0$ such  that $Y$ is 
$\delta_1$--hyperbolic. Moreover, there exists $H>0$ such that for any vertices $x,y$ of $X$ the 
path $f([x,y])$ is $H$--Hausdorff close to any geodesic $[f(x), f(y)]$ in $Y$.
\end{cor}

\begin{proof}
For every pair of vertices $a,b\in X$, let $g_{a,b}$ be any geodesic segment 
$[a,b]$ and let $\CG$ be the set of all these paths. Also. for any vertices $a,b,c$ 
of $X$ let 
\[
\Phi(a,b,c) = \Phi(b,c,a) =\Phi(c,a,b)
\] 
be any vertex of $X$ that is at most $\delta_0$ away from each of 
$[a,b], [b,c], [a,c]$. The hyperbolicity of $X$ implies that 
$(\mathcal G, \Phi)$ forms a $(b_1,b_2)$--thin triangles structure on $X$
for some $b_1$ and $b_2$ depending on $\delta_0$. 
We now \emph{push} this structure $(\mathcal G, \Phi)$ forward via 
the map $f$.

For any vertex $y$ of $Y$ choose a vertex $v_y$ of $X$ such that $f(v_y)=y$.  

For any vertices $y, z\in Y$ choose a geodesic $g_{v_y, v_z} $ from $v_y$ to $v_z$ in $X$ (note that such a geodesic is generally not unique) and let $g_{y,z}':=f(g_{v_y, v_z} )$. Then let 
$\CG'=\{g_{y,z}'| y, z \in V(Y)\}$.
Now, for any vertices $w,y,z \in Y$ put 
\[
\Phi'(w,y,z):=f( \Phi(v_w,v_y,v_z)).
\]
We claim that, the pair $(\CG', \Phi')$ is a $(B_1, B_2)$--thin triangles 
structure for $Y$ for some $B_1$ and $B_2$. 
The conditions (1) and (3) of Definition~\ref{defn:thin} are satisfied 
as long as $B_2 \geq L\, b_2$ since $f$ is $L$--Lipschitz. 
Thus we only need to verify that condition (2) of Definition~\ref{defn:thin} 
holds for $(\CG', \Phi')$.

Let $y,z$ be vertices of $Y$, $v_y, v_z$ be the associated
vertices in $X$, $g_{v_y, v_z} \from [0,l] \to X$ be the path
in $\CG$ connecting $v_x$ to $v_y$ and $g'_{y,z} \from [0,l'] \to Y$
be the $f$-image of $g_{v_y, v_z}$. In the interest of brevity, we denote these paths 
simply by $g$ and $g'$. 

Let $B_1 = M_1$ and let, $a,b \in Y$ and $s', t' \in [0,l']$ be such that
\[
d_Y(a, g'(s')) \le B_1
\qquad\text{and}\qquad
d_Y(b, g'(t')) \le B_1.
\]
We need to bound the Hausdorff distance between 
$g_{a,b}'=f(g_{v_a,v_b})$ and $g'\big|_{[s',t']}$.

Let $s, t \in [0,l]$ be such that $f g(s) = g'(s')$ and $f g (t) = g'(t')$.
Let $u$ be a vertex of $g_{a,b}'$. From hyperbolicity, we have
$v_u$ is contained in a $2\delta_0$--neighborhood of the union
\[
g\big|_{[s,t]}\cup [g(s), v_a] \cup [g(t), v_b].
\]
Thus $u$ is $(2 L \, \delta_0)$--close to the union 
\[
g'\big|_{[s',t']} \cup f([g(s), v_a]) \cup f([g(t), v_b]).
\]
But the $d_Y$--diameters of $f([g(s), v_a])$ $f([g(t), v_b])$ are less than $M_2$. 
Hence, $u$ is in $(2 L \, \delta_0 + M_2)$--neighborhood of $g'\big|_{[s',t']}$.
Similarly, $g'\big|_{[s',t']}$ is in the same size neighborhood of
$g'_{a,b}$. The condition (2) of Definition~\ref{defn:thin}
holds for $B_2 = (2 L \, \delta_0 + M_2)$.

Therefore, by Proposition~\ref{prop:bow}, the graph $Y$ is $\delta_1$-hyperbolic, 
and, moreover, for any two vertices $y,z$ of $Y$ the path $g'_{y,z}=f(g_{v_y,v_z})$ 
is $H'$--Hausdorff close to $[y,z]$ for some constant $H'\ge 0$ independent 
of $y,z$. Since $g_{v_y,v_z}$ was chosen to be a geodesic from $v_y$ to $v_z$ in $X$ and any two such geodesics are $\delta_0$-Hausdorff close, by increasing the constant $H'$ we also get that for any geodesic $[y,z]$ from $y$ to $z$ in $Y$ and any geodesic $[v_y,v_z]$ from $v_y$ to $v_z$ in $X$ the paths $[y,z]$ and $f([v_y,v_z])$ are $H'$-Hausdorff close in $Y$.  

Now let $y,z$ be any vertices of $Y$ and let $v_y',v_z'$ be arbitrary vertices of $X$ such that $f(v_y')=y$ and $f(v_z')=z$. A geodesic $[v_y',v_z']$ in $X$ is contained in the $2\delta_0$-neighborhood of $[v_y',v_y]\cup [v_y,v_z]\cup [v_z,v_z']$. Since $f(v_y)=f(v_y')=y$ and $f(v_z)=f(v_z')=z$, the assumptions on $f$ imply that $f([v_y',v_y])$ is contained in the $M_2$-ball around $y$ and  $f([v_z,v_z'])$ is contained in the $M_2$-ball around $z$ in $Y$. Moreover, we have already shown that $f([v_y,v_z])$ is $H'$-Hausdorff close to $[y,z]$. Therefore $f([v_y',v_z'])$ is contained in the $H$-neighborhood of $[y,z]=[f(v_y'),f(y_z')]$ with $H=2L\delta_0+M_2+H'$.  A similar argument shows that $[y,z]=[f(v_y'),f(y_z')]$ is contained in the $H$-neighborhood of  $f([v_y',v_z'])$. Thus $[y,z]=[f(v_y'),f(v_z')]$ and $f([v_y',v_z'])$ are $H$-Hausdorff close, as required.

\end{proof}

\begin{prop}\label{prop:image}
For any positive integers $\delta_0$, $L$, $M$ and $D$ there exist $\delta_1>0$ and $H>0$ so 
that the following holds. 

Let $X$, $Y$ be connected graphs, such that $X$ is 
$\delta_0$--hyperbolic. Let  $f \from X\to Y$ be a $L$--Lipschitz graph map 
for some $L\ge 0$. Let $S\subseteq V(X)$ be such that:
\begin{enumerate}
\item We have $f(S)=V(Y)$.
\item The set $S$ is $D$--dense in $X$ for some $D>0$.
\item For $x,y\in S$, if $d(f(x),f(y))\le 1$ 
then for any geodesic $[x,y]$ in $X$ we have 
\[
\diam_Y( f([x,y]))\le M.
\]
\end{enumerate}
Then $Y$ is $\delta_1$--hyperbolic and, for any $x,y\in V(X)$ and any 
geodesic $[x,y]$ in $X$, the path $f([x,y])$ is $H$--Hausdorff close to any 
geodesic $[f(x),f(y)]$  in $Y$.
\end{prop}
\begin{proof}
First we show that, for every $m_1>0$, there is $m_2>0$ so that 
whenever $x,y\in S$ satisfy $d_Y(f(x),f(y))\le m_1$ then 
$\diam_Y( f([x,y]))\le m_2$. Indeed let $x,y\in S$ be as above and consider a 
geodesic path $[f(x),f(y)]$ in $Y$. Let 
\[
f(x)=z_0, z_1, \dots, z_t=f(y), \qquad t <m_1
\]
be the sequence of consecutive vertices on $[f(x),f(y)]$. 
Let $x_0=x$, $x_t=y$ and for $1\le i\le t-1$ let $x_i\in S$ be 
such that $f(x_i)=z_i$. Such $x_i$ exist since by assumption $f(S)=V(Y)$. 
We have $\diam_Y f([x_i, x_{i+1}])\le M$. 
By hyperbolicity, the geodesic $[x,y]$ is contained in the 
$(m_1\delta_0)$--neighborhood of the union 
\[
\bigcup_{i=0}^{t-1} \, [x_i, x_{i+1}].
\]
Since $f$ is $L$--Lipschitz, $f([x,y])$ is contained in the 
$(L m_1 \delta_0)$--neighborhood of  
\[
\bigcup_{i=0}^{t-1} f\big([x_i, x_{i+1}]\big).
\]
But each $f([x_i, x_{i+1}])$ has diameter $\le M$. Therefore, 
$f([x,y])$ has a diameter of at most $m_2=(m_1 M+2 L m_1 \delta_0)$. 

Now let $M_1\ge 0$ and $x,y\in V(X)$ be arbitrary vertices with
$d_Y(f(x),f(y))\le M_1$.  Since $S$ is $D$--dense in $X$, there exist 
$x',y'\in S$ such that $d(x,x'), d(y,y')\le D$.  The fact that $f$ is $L$--Lipschitz 
implies that $d(f(x'),f(y'))\le M_1+2DL$. Therefore, by the above claim,
it follows that the
\[
\diam_Y f([x',y']) \le m_2(M_1+2DL).
\]
Since $X$ is $\delta_0$--hyperbolic and $d(x,x'), d(y,y')\le D$, 
we have that $[x,y]$ and $[x',y']$ are $(2\delta+2D)$--Hausdorff close. 
Again, using that $f$ is $L$--Lipschitz, we conclude that $f([x,y])$ 
has a diameter of at most 
\[
M_2= m_2(M_1+2DL)+ 4L(2\delta_0+2D).
\]
The assumption of Corollary~\ref{cor:contraction} are now satisfied for
constants $\delta_0$, $L$, $M_1$ and $M_2$.
Proposition~\ref{prop:image} now follows from Corollary~\ref{cor:contraction}.
\end{proof}

Proposition~\ref{prop:image} easily implies the well-known fact that ``coning-off" or ``electrifying" a family of uniformly quasiconvex subsets in a hyperbolic space produces a hyperbolic space.
Various versions of this statement have multiple appearances in the literature; see, for example, Lemma 4.5 and Proposition~4.6 in \cite{Farb98}, Proposition~7.12 in \cite{Bow97}, Lemma~2.3 in \cite{MjR}, Theorem~3.4 in~\cite{DG},  etc.  Maher and Schleimer~\cite{MahS} appear to be the first once to explicitly note that after ``electrifying'' a family of uniformly quasiconvex subsets in a hyperbolic space,  not only is the resulting space again hyperbolic, but the image of a geodesic is a reparameterized quasigeodesic.

We give a version of the ``coning-off'' statement here phrased in the context of graphs with simplicial metrics.

\begin{prop}\label{prop:qc}
Let $X$ be a connected graph with simplicial metric $d_X$ such that $(X,d_X)$ is $\delta_0$-hyperbolic.
Let $C>0$ and $(X_j)_{j\in J}$ be a family of subgraphs of $X$ such that each $X_j$ is a $C$-quasiconvex subset of $X$.
Let $Y$ be the graph obtained from $X$ by adding to $X$ the new edges $e_{x,y,j}$ with endpoints $x$, $y$ whenever $j\in J$ and $x,y$ are vertices of $X_j$ (thus $X$ is a subgraph of $Y$).  Let $d_Y$ be the simplicial metric on $Y$.

Then $Y$ is $\delta_1$-hyperbolic for some constant $\delta_1>0$ depending only on $C$ and $\delta_0$. Moreover there exists $H=H(C,\delta_0)>0$ such that whenever $x,y\in V(X)$, $[x,y]_X$ is a $d_X$-geodesic from $x$ to $y$ in $X$ and $[x,y]_Y$ is a $d_Y$-geodesic from $x$ to $y$ in $Y$ then $[x,y]_X$ and $[x,y]_Y$ are $H$-Hausdorff close in $(Y,d_Y)$.
\end{prop}

\begin{proof}
Let $f:X\to Y$ be the inclusion map and put $S=V(X)$. 
We claim that the conditions of Proposition~\ref{prop:image} are satisfied.

By construction $V(X)=V(Y)$, so $f(S)=V(Y)$. Also, the map $f$ is obviously 1-Lipschitz.
Suppose now that $x,y\in V(X)$ are such that $d_Y(x,y)\le 1$. If $d_Y(x,y)<1$ then $d_Y(x,y)=0$, so that $x=y$. In this case it is obvious that condition~(3) of Proposition~\ref{prop:image} holds for $f([x,y]_X)$. Suppose now that $d_Y(x,y)=1$. Thus there exists an edge $e$ in $Y$ with endpoints $x,y$. If $e$ is an edge of $X$ then $d_X(x,y)=1$ and it is again obvious that condition~(3) of Proposition~\ref{prop:image} holds for $f([x,y]_X)$.  Suppose now that $e=e_{x,y,j}$ for some $j\in J$. Thus $x,y\in V(X_j)$. Since $X_j$ is $C$-quasiconvex in $X$, we see that for any point $u$ on $[x,y]_X$ there exists a vertex $z$ of $X_j$ with $d_X(u,z)\le C$. Hence $d_Y(u,z)\le C$ as well. We have $d_Y(x,z)\le 1$ and $d_Y(z,y)\le 1$ since $x,y,z$ are vertices of $X_j$. Hence $d_Y(f(u),x)=d_Y(u,x)\le C+1$. Thus $f([x,y]_X)$ is contained in the $d_Y$-ball of radius $C+1$ centered at $x$ in $(Y,d_Y)$, and again condition~(3) of Proposition~\ref{prop:image} holds.

Thus Proposition~\ref{prop:image}  applies and the conclusion of Proposition~\ref{prop:qc} follows.

\end{proof}

Note that the assumptions of  Proposition~\ref{prop:qc} do not require the family $(X_j)_{j\in J}$  to be ``sufficiently separated". Such a requirement is present in many versions of Proposition~\ref{prop:qc} available in the literature, although this assumption is not in fact necessary and, in particular,  Proposition~7.12 in \cite{Bow97} does not impose the ``sufficiently separated" requirement.
We have derived Proposition~\ref{prop:qc} from Proposition~\ref{prop:image}, which in turn was a consequence of Corollary~\ref{cor:contraction}. A close comparison of these statements show that the converse implication does not work, and that Corollary~\ref{cor:contraction} is a more general statement than Proposition~\ref{prop:qc}.

\section{Free factor complex and free splitting complex}

\begin{defn}[Free factor complex]\label{defn:FF_N}
Let $F_N$ be a free group of finite rank $N\ge 3$.

The \emph{free factor complex} $FF_N$ of $F_N$ is a simplicial complex
defined as follows. The
set of vertices $V(FF_N)$ of $FF_N$ is defined as the set of all
$F_N$--conjugacy classes $[A]$ of proper free factors $A$ of $F_N$.
Two distinct vertices $[A]$ and $[B]$ of $FF_N$ are joined by an edge
whenever there exist proper free factors $A,B$ of $F_N$ representing
$[A]$ and $[B]$ respectively, such that either $A\le B$ or $B\le A$.

More generally, for $k\ge 1$, a collection of $k+1$ distinct vertices
$[A_0], \dots, [A_k]$ of $FF_N$ spans a $k$--simplex in $FF_N$ if, up
to a possible re-ordering of these vertices there exist
representatives $A_i$ of $[A_i]$ such that $A_0\le A_1\le \dots \le
A_k$. 

There is a canonical action of $\Out(F_N)$ on $FF_N$ by simplicial
automorphisms: If  $\Delta=\{[A_0], \dots, [A_k]\}$ is a $k$ simplex
and $\phi\in \Out(F_N)$, then $\phi(\Delta):=\{[\phi(A_0)], \dots, [\phi(A_k)]\}$.
\end{defn}

It is not hard to check that for $N\ge 3$ the complex $FF_N$ is
connected, has dimension $N-2$ and that $FF_N/\Out(F_N)$ is compact.

\begin{defn}[Free splitting complex]\label{defn:FS_N}
Let $F_N$ be a free group of finite rank $N\ge 3$.

The \emph{free splitting complex} $FS_N$ is a simplicial complex
defined as follows. The vertex set $V(FS_N)$ consists of
equivalence classes of splittings $F_N=\pi_1(\mathbb A)$, where
$\mathbb A$ is a graph of groups with a single topological edge $e$ (possibly a loop
edge) and the trivial edge group such that the action of $F_N$ on the
Bass-Serre tree $T_\mathbb A$ is minimal (i.e. such that if $e$ is a
non-loop edge then both vertex groups in $\mathbb A$ are
nontrivial). Two such splittings $F_N=\pi_1(\mathbb A)$ and
$F_N=\pi_1(\mathbb B)$ are equivalent if there exists an
$F_N$--equivariant isometry between $T_\mathbb A$ and $T_\mathbb B$.
We denote the equivalence class of a splitting $F_N=\pi_1(\mathbb A)$
by $[\mathbb A]$.

The edges in $FS_N$ correspond to two splittings admitting a common
refinement. Thus two distinct vertices  $[\mathbb A]$ and $[\mathbb
B]$ of $FS_N$ are joined by an edge whenever there exists a splitting
$F_N=\pi_1(\mathbb D)$ such that the graph of groups $\mathbb D$ has
exactly two topological edges, both with trivial edge groups, and such that
collapsing one of these edges produces a splitting of $F_N$ representing
$[\mathbb A]$ and collapsing the other edge produces a splitting
representing $[\mathbb B]$.

More generally, for $k\ge 1$ a collection of $k+1$ distinct vertices
$[\mathbb A_0], \dots, [\mathbb A_k]$ of $FS_N$ spans a $k$--simplex in
$FS_N$ whenever there exists a splitting $F_N=\pi_1(\mathbb D)$ such
that the graph of groups $\mathbb D$ has the following properties:

(a) The underlying graph of $\mathbb D$ has exactly $k+1$ topological
edges, $e_0,\dots, e_k$.
\smallskip

(b) The edge group of each $e_i$ is trivial.

\smallskip

(c) For each $i=0,\dots, k$ collapsing all edges except for $e_i$ in
$\mathbb D$ produces a splitting of $F_N$ representing $[\mathbb
A_i]$.

\medskip

The complex $FS_N$ comes equipped with a natural action of $\Out(F_N)$
by simplicial automorphisms.

\end{defn}

Again, it is not hard to check that for $N\ge 3$ the complex $FS_N$ is
finite-dimensional, connected and that the quotient $FS_N/F_N$ is
compact.

We denote the barycentric subdivision of $FS_N$ by $FS_N'$.

\begin{defn}[Marking]
Let $N\ge 2$.
Recall that a \emph{marking} on $F_N$ is an isomorphism  
$\alpha\from  F_N\to \pi_1(\Gamma, v)$ where $\Gamma$ is a finite connected graph without
any degree-one and degree-two vertices and $v$ is a vertex of $\Gamma$. 
By abuse of notation, if $\alpha$ is specified, we will often refer to
$\Gamma$ as a marking.

Two markings $\alpha\from  F_N\to \pi_1(\Gamma, v)$ and $\alpha'\from  F_N\to
\pi_1(\Gamma', v')$ are said to be equivalent, if there exists an
$F_N$--equivariant isometry $\widetilde {(\Gamma, v)}\to \widetilde
{(\Gamma', v')}$. 
The equivalence class of a marking $\alpha\from  F_N\to \pi_1(\Gamma,v)$ is
denoted by $[\alpha]$ or, if $\alpha$ is already specified,
just $[\Gamma]$.
\end{defn}

\begin{conv}[Barycenters]\label{conv:bar}
Note that for $N\ge 3$ any marking $\alpha\from  F_N\to \pi_1(\Gamma)$ corresponds
to a simplex $\Delta_\alpha$ in $FS_N$, as follows. We can view $\Gamma$ as a
graph of groups by assigning trivial groups to all the vertices and
edges of $\Gamma$. Then the vertices of $\Delta_\alpha$ correspond to the
(topological) edges of $\Gamma$ and come from choosing an edge $e$ of
$\Gamma$ and collapsing all the other edges of $\Gamma$.
It is easy to see that $\Delta_\alpha$ depends only on the equivalence
class $[\alpha]$ of the marking $\alpha$.

We denote the vertex of $FS_N'$ given by the barycenter of
$\Delta_\alpha$ by $z(\alpha)$ or, if it is more convenient, by
$z(\Gamma)$. Note that if $[\alpha]=[\beta]$ then   $z(\alpha)=z(\beta)$. 
We will sometimes refer to a marking $\Gamma$ as a vertex of $FS_N'$; when that happens, we always mean the vertex $z(\Gamma)$.
\end{conv}

\section{The free bases graph}

If $\Gamma$ is a graph (i.e. a one-dimensional CW-complex), then any
topological edge (i.e. a closed 1-cell) of $\Gamma$ is homeomorphic to
either $[0,1]$ or to $\mathbb S^1$ and
thus admits exactly two orientations. An \emph{oriented edge} of
$\Gamma$ is a topological edge together with a choice of an
orientation on this edge. If $e$ is an oriented edge of $\Gamma$, we denote by
$e^{-1}$ the oriented edge obtained by changing the orientation on $e$
to the opposite one. Note that $(e^{-1})^{-1}=e$ for any oriented edge
$e$. For an oriented edge $e$ we denote the initial vertex of $e$ by
$o(e)$ and the terminal vertex of $e$ by $t(e)$. Then $o(e^{-1})=t(e)$
and $t(e^{-1})=o(e)$. We will denote by $E\Gamma$ the set of oriented
edges of $\Gamma$ and by $V\Gamma$ the set of vertices of $\Gamma$.

Let $N\ge 2$. We denote by $W_N$ the graph with a single vertex $v_0$ and
$N$ distinct oriented loop-edges $e_1,\dots, e_N$.

\begin{defn}[$\mathcal A$--rose]
Let $\mathcal A=\{a_1,\dots, a_N\}$ be a free basis of $F_N$. Define the
\emph{$\mathcal A$--rose} $R_\mathcal A$ as the marking
$\alpha_\mathcal A\from  F_N\to \pi_1(W_N, v_0)$
where $\alpha_\mathcal A$ sends $a_i$ to the loop at $v_0$ in $\Gamma$
corresponding to $e_i$, traversed in the direction given by
the orientation of $e_i$.
\end{defn}

\begin{defn}[Free bases graph]\label{defn:FB_N}
Let $N\ge 3$. The \emph{free bases graph} $FB_N$ of $F_N$ is a simple graph defined as follows. The vertex
set $V(FB_N)$ consists of equivalence classes free bases $\mathcal A$
of $F_N$. Two free bases $\mathcal A$ and $\mathcal B$ of $F_N$ are
considered equivalent if the Cayley graphs
$T_\mathcal A$ and $T_\mathcal B$ of
$F_N$ with respect to $\mathcal A$ and $\mathcal B$ are
$F_N$--equivariantly isometric. We denote the equivalence class of a
free basis $\mathcal A$ of $F_N$ by $[\mathcal A]$.

Note that for free bases $\mathcal A=\{a_1,\dots, a_N\}$ and $\mathcal
B=\{b_1,\dots, b_N\}$ of $F_N$ we have  $[\mathcal A]=[\mathcal B]$ if and only if there exist a
permutation $\sigma\in S_N$, an element $g\in F_N$ and numbers
$\epsilon_i\in \{1, -1\}$ (where $i=1,\dots, N$) such that
\[
b_i= g^{-1} a_{\sigma(i)}^{\epsilon_i} g
\] 
for $i=1,\dots, N$. Thus $[\mathcal A]=[\mathcal B]$ if and only if
the roses $R_\mathcal A$ and $R_\mathcal B$ are equivalent as markings.
Note also that for any free basis $\mathcal A$ of
$F_N$ and any $g\in F_N$ we have $[g^{-1}\mathcal Ag]=[\mathcal A]$.

The edges in  $FB_N$ are defined as follows. Let  $[\mathcal A]$ and
$[\mathcal B]$ be two distinct vertices of $FB_N$. These vertices are
adjacent in $FB_N$  whenever there exists $a\in \mathcal A$ such that some element $b\in \mathcal B$ is conjugate to $a$ or
$a^{-1}$. Thus two distinct vertices $v_1,v_2$ of $FB_N$ are adjacent if and only if there exist free bases $\mathcal A$ and $\mathcal B$ representing $v_1$, $v_2$ accordingly such that $\mathcal A\cap \mathcal B\ne \emptyset$.

The graph $FB_N$ comes equipped with a natural $\Out(F_N)$ action by simplicial automorphisms.

\end{defn}

\begin{prop} \label{prop:FB}

Let $N\ge 3$. Then:

\begin{enumerate}
\item The graph $FB_N$ is connected.

\item For each vertex $v=[\mathcal A]$ of $FB_N$ choose some
  $a_v\in \mathcal A$. Consider the map 
  \[
  h:V(FB_N) \to V(FF_N)
  \]
  defined as $h(v)=[\langle a_v\rangle]$ for every vertex $v$ of
  $FB_N$. Extend $h$ to a graph-map 
  \[
  h\from  FB_N\to FF_N
  \] 
  by sending every edge $e$ of $FB_N$ with endpoints $v,v'$ to a geodesic path
  $[h(v),h(v')]$ in $FF_N^{(1)}$. Then:

(a) The map $h$ is a quasi-isometry. In particular, the
  complexes $FB_N$ and $FF_N$ are quasi-isometric.

(b) The set $h(V(FB_N))$ is $3$--dense in $FF_N^{(1)}$

\end{enumerate}

\end{prop}

\begin{proof}
First we will show that $h$ is $4$--Lipschitz.
Since $h\from FB_N\to FF_N^{(1)}$ is a graph-map, it suffices to check that
for any two adjacent vertices $v, v'$
of $FB_B$ we have $d_{FF_N}(h(v), h(v'))\le 4$.

Let $v=[\mathcal A]$ and $v'=[\mathcal B]$ be two adjacent vertices of
$FB_N$. Hence we may choose free bases $\mathcal A$ representing $v$
and $\mathcal B$ representing $v'$ such that $\mathcal A\cap \mathcal
B\ne \emptyset$. Up to re-ordering these bases, we may assume that $\mathcal
A=\{a_1,\dots, a_N\}$, $\mathcal B=\{b_1,\dots, b_N\}$ and that $a_1=b_1$.
Then $a_v=a_i$ and $a_{v'}=b_j$ for some $1\le i,j\le N$ and
thus, by definition of $h$, we have $h(v)=[\langle a_i\rangle]$,
$h(v')=[\langle b_j\rangle]$.

We will assume that $i>1$ and $j>1$ as the cases where $i=1$ or $j=1$
are easier. Then in $FF_N$ we have 
\[
d_{FF_N}\Big([\langle a_i\rangle], [\langle a_i, a_1\rangle]\Big)=
d_{FF_N}\Big([\langle a_i, a_1\rangle], [\langle a_1\rangle]\Big)=1
\]
and
\[
d_{FF_N}\Big([\langle b_j\rangle], [\langle b_j, b_1\rangle]\Big)=
d_{FF_N}\Big([\langle b_j, b_1\rangle], [\langle b_1\rangle]\Big)=1.
\]
Since $a_1=b_1$, by the triangle inequality we conclude that 
\[
d_{FF_N}(h(v),h(v'))=d_{FF_N}([\langle a_i\rangle], [\langle b_j\rangle])\le 4.
\]
Thus the map $h$ is $4$--Lipschitz, as claimed.

To show that $h$ is a quasi-isometry we will construct a
``quasi-inverse'', that is a Lipschitz map $q\from FF_N^{(1)}\to FB_N$
such that there exists $C\ge 0$ with the property that for every
vertex $v$ of $FB_N$ $d_{FB_N}(v, q(h(v)))\le C$ and that for every vertex
$u$ of $FF_N$, $d_{FN_N}(u,h(q(u)))\le C$.

We define $q$ on $V(FF_N)$ and then extend $q$ to edges in a
natural way, by sending every edge to a geodesic joining the images of
its end-vertices.

Let $u=[K]$ be an arbitrary vertex of $FF_N$ (so that $K$ is a proper
free factor of $F_N$). We choose a free basis $\mathcal B_K$ of $K$
and then a free basis $\mathcal A_K$ of $F_N$ such that $\mathcal
B_K\subseteq \mathcal A_K$. Put $q(u)=[\mathcal A_K]$.

First we check that $q$ is Lipschitz. Let $u=[K]$ and $u'=[K']$ be
adjacent vertices of $FF_N$, where $K, K'$ are proper free factors of
$F_N$. We may assume that $K\le K'$ is a proper free factor of
$K'$. Since $K'\ne F_N$, there exists $t\in \mathcal A_{K'}\setminus
\mathcal B_{K'}$. Since $K$ is a free factor of $K'$, we can find a
free basis $\mathcal A$ of $F_N$ such that $t\in \mathcal A$ and
$\mathcal B_K\subseteq \mathcal A$. Since $t\in \mathcal A_{K'}\cap
\mathcal A$, we have $d([\mathcal
A_{K'}], [\mathcal A])\le 1$ in $FB_N$.  Since $\mathcal B_K\subset
\mathcal A\cap \mathcal A_K$, it follows that $d([\mathcal
A_{K}], [\mathcal A])\le 1$ in $FB_N$. Therefore 
\[
d_{FB_N}\big(q(u), q(u')\big)=
d_{FB_N}\big([\mathcal A_{K}],[\mathcal A_{K'}]\big)\le 2.
\]
Hence $q$ is $2$--Lipschitz.

For a vertex $v=[\mathcal A]$ of $FB_N$ let us now estimate 
$d_{FB_N}(v,q(h(v)))$. We have $h(v)=[\langle a_v\rangle]$ for some $a_v\in
\mathcal A$. The group $K=\langle a_v\rangle$ is infinite cyclic (that
is free of rank $1$). Therefore this group has only two possible free
bases, $\{a_v\}$ and $\{a_v^{-1}\}$. We will assume that $\mathcal
B_K=\{a_v\}$ as the case  $\mathcal B_K=\{a_v^{-1}\}$ is similar. Then, 
by definition, $\mathcal A_K$ is a
free basis of $F_N$ containing $a_v$ and $q(h(v))=q([K])=[\mathcal A_K]$.
Thus $a_v\in \mathcal A\cap \mathcal A_K$ and hence 
$d_{FB_N}(v, q(h(v)))\le 1$ in $FB_N$.

Now let $u=[K]$ be an arbitrary vertex of $FF_N$. We need to estimate
$d_{FF_N}(u, h(q(u)))$. By definition, $v:=h(u)=[\mathcal A_K]$ where 
$\mathcal A_K$ is a free basis of $F_N$ containing as a (proper) subset a free
basis $\mathcal B_K$ of $K$. Then $a_v\in \mathcal A_K$ and
$h(q(u))=h(v)=[\langle a_v\rangle]$. Choose an element $b\in \mathcal
B_K$. It may happen that $b=a_v$, but in any case $K':=\langle
b,a_v\rangle$ is a proper free factor of $F_N$. Then 
\[
d_{FF_N}\big([K], [\langle b \rangle]\big)\le 1, \quad 
d_{FF_N}\big([\langle b \rangle], [K']\big)\le 1
\quad\text{and}\quad 
d_{FF_N}\big([K'], [\langle a_v\rangle]\big)\le1.
\] 
Therefore
\[
d_{FF_N}\big(u, h(q(u)\big)=d_{FF_N}\big([K],[\langle a_v\rangle]\big)\le 3.
\]
Thus indeed $q$ is a quasi-inverse for $h$, and hence $h$ is a
quasi-isometry, as required.

We next show that $h(V(FB_N))$ is $3$--dense in $FF_N^{(1)}$.
Indeed, let $K\le F_N$ be an arbitrary proper free factor of
$F_N$. Let $a_1,\dots, a_m$ (where $1\le m<N$) be a free basis of $K$
and choose $a_{m+1}\dots, a_N$ such that $\mathcal A=\{a_1,\dots, a_N\}$ is 
a free basis of $F_N$. Then $h([\mathcal A])=[\langle
a_i\rangle]$ for some $1\le i\le N$. In $FF_N$ we have 
\[
d_{FF_N}\big([K], [\langle a_1\rangle]\big)\le 1, \quad 
d_{FF_N}\big([\langle a_1\rangle], [\langle a_1, a_i\rangle]\big)\le 1
\quad\text{and}\quad 
d_{FF_N}\big([\langle a_1, a_i\rangle], [\langle a_i\rangle]\big)\le 1.
\] 
Since $h([\mathcal A])=[\langle
a_i\rangle]$, it follows that $d_{FF_N}([K], h(v))\le 3$. Thus indeed
$h(V(FB_N))$ is $3$--dense in $FF_N^{(1)}$, as claimed.
\end{proof}

\begin{defn}[Free basis defined by a marking]\label{defn:BT}
If $\alpha\from  F_N\to \pi_1(\Gamma,v)$ is a marking, and $T\subseteq
\Gamma$ is a maximal tree in $\Gamma$, there is a naturally associated
free basis $\mathcal B(\alpha, T)$ (which we will also sometimes
denote  $\mathcal B(\Gamma, T)$) of $F_N$. Namely, in this case
$\Gamma-T$ consists of $N$ topological edges. Choose oriented edges
$e_1,\dots, e_N\in E(\Gamma-T)$ so that $E(\Gamma-T)=\{e_1^{\pm 1},
\dots, e_N^{\pm 1}\}$. For $j=1,\dots, N$ put 
\[
\gamma_j=[v,o(e_j)]_T \, e_j \, [t(e_j),v]_T,
\] 
where $[u,u']_T$ denotes the (unique) geodesic in the tree $T$ from $u$ to $u'$ for $u,u'\in V(\Gamma)=V(T)$.

Then $\gamma_1,\dots, \gamma_N$ is a free basis of
$\pi_1(\Gamma, v)$. Put $\mathcal B(\alpha,
T):=\{\alpha^{-1}(\gamma_1), \dots , \alpha^{-1}(\gamma_1)\}$.
\end{defn}

\begin{rem}\label{rem:BT}
One can show that there is a constant $C=C(N)>0$ such that if $\alpha\from  F_N\to \pi_1(\Gamma,v)$ and
$\alpha'\from  F_N\to \pi_1(\Gamma',v')$ are equivalent markings and
$T\subseteq \Gamma$, $T'\subseteq \Gamma'$ are maximal trees, then
\[
d_{FB_N}\Big([\mathcal B(\alpha, T)], [\mathcal B(\alpha', T')]\Big)\le C.
\]
This can be shown, for example, using the quasi-isometry $q:FB_N\to FF_N$ constructed in Proposition~\ref{prop:FB}. Thus the definitions imply that if $T$ is a maximal tree in $\Gamma$ and $e$ is an edge of $\Gamma\setminus T$, then $q([\mathcal B(\alpha, T)])$ is a bounded distance away in $FF_N$ from the free factor of $F_N$ corresponding to any of the vertex groups in the graph of groups $\Gamma_e$ obtained by collapsing $\Gamma\setminus e$. On the other hand, for any two edges $e_1, e_2$ of $\Gamma$ the free splittings $\Gamma_{e_1}$ and $\Gamma_{e_2}$ are adjacent vertices of $FS_N$ and therefore (e.g. using the Lipschitz map $\tau:FS_N\to FF_N$ from the Introduction),  any two vertex groups $A_1$ and $A_2$ from these splittings are bounded distance away in $FF_N$. 
\end{rem}

\section{$\mathcal A$--graphs and Stallings folds}

We briefly discuss here the language and machinery of Stallings
foldings, introduced by Stallings in a seminal paper~\cite{St}. We
refer the reader to \cite{KM} for detailed background on the topic.

If $\Gamma$ is a finite connected non-contractible graph, we denote by
$\Core(\Gamma)$ the unique minimal subgraph of $\Gamma$ such that the
inclusion $\Core(\Gamma)\subseteq \Gamma$ is a homotopy
equivalence. Thus $\Core(\Gamma)$ carries $\pi_1(\Gamma)$ and we can
obtain $\Gamma$ from $\Core(\Gamma)$ by attaching finitely many trees.

\begin{defn}[$\mathcal A$--graph]
Let $\mathcal A$ be a free basis of $F_N$ and let $R_\mathcal A$ be
the corresponding rose marking.
An \emph{$\mathcal A$--graph} is a graph $\Gamma$ with a
labelling function $\mu\from  E\Gamma\to \mathcal A^{\pm 1}$ (where
$E\Gamma$ is the set of oriented edges of $\Gamma$) such that for
every oriented edge $e\in E\Gamma$ we have
$\mu(e^{-1})=\mu(e)^{-1}$.
  
Note that there is an obvious way to view the rose $R_\mathcal A$ as an 
$\mathcal A$--graph. 
 Any $\mathcal A$--graph $\Gamma$ comes equipped with a canonical
label-preserving graph-map $p\from  \Gamma\to R_\mathcal A$ which sends all
vertices of $\Gamma$ to the (unique) vertex of $R_\mathcal A$ and
which sends every oriented edge of $\Gamma$ to the oriented edge of
the rose $R_\mathcal A$ with the same label. We call $p$ the
\emph{natural projection}.

Let $\Gamma$ be a finite connected $\mathcal A$-graph containing at least one vertex of degree $\ge 3$.
Following Handel-Mosher~\cite{HM}, we call vertices of $\Gamma$ that have degree $\ge 3$ \emph{natural vertices}. The complement of the set of natural vertices in $\Gamma$ consists of a disjoint union of intervals whose closures, again following~\cite{HM}, we call \emph{natural edges}.

\end{defn}

Recall that in the definition of a marking on $F_N$ the graph
appearing in that definition had no degree-one and degree-two vertices.

\begin{rem}\label{rem:A}
Suppose that $\Gamma$ is a connected $\mathcal A$--graph such that the
natural projection $\Gamma\to R_\mathcal A$ is a homotopy equivalence.
Then the projection $p\from  \Core(\Gamma)\to R_\mathcal A$ is a homotopy equivalence.

Then, via using the homotopy inverse of $p$ and  making inverse
subdivisions in $\Core(\Gamma)$ to erase all the degree-2 vertices, we
get an actual marking of $F_N$, $\alpha\from  F_N\to \pi_1(\overline{\Gamma})$. Here $\overline \Gamma$ is the graph obtained from
  $\Core(\Gamma)$ by doing inverse edge-subdivisions to erase all
  degree-two vertices. In this case we call $\alpha$ the \emph{marking
    associated with $\Gamma$} and denote $\alpha$ by $\alpha_\Gamma$,
  or, sometimes just by $\overline \Gamma$.
\end{rem}  

\begin{defn}[Folded graphs and Stallings folds]
Let $\Gamma$ be an $\mathcal A$--graph. We say that $\Gamma$ is
\emph{folded} if there does not exist a vertex $v$ of $\Gamma$ and two
distinct oriented edges $e_1,e_2$ with $o(e_1)=o(e_2)=v$ such that
$\mu(e_1)=\mu(e_2)$. Otherwise we say that $\Gamma$ is \emph{non-folded}.

Let $\Gamma$ be a non-folded $\mathcal A$--graph, let $e_1,e_2$ be two
distinct oriented edges of $\Gamma$ such that $o(e_1)=o(e_2)=v\in
V(\Gamma)$ and such that $\mu(e_1)=\mu(e_2)=a\in \mathcal A^{\pm
  1}$. Construct an $\mathcal A$--graph $\Gamma'$ by identifying the
edges $e_1$ and $e_2$ into a single edge $e$ with label $\mu(e)=a$.
We say that
$\Gamma'$ is obtained from $\Gamma$ by a \emph{Stallings fold}.
In this case there is also a natural label-preserving \emph{fold map}
$f\from \Gamma\to \Gamma'$. It is easy to see that the fold map $f$ is a
homotopy equivalence if and only if $t(e_1)\ne t(e_2)$ in $\Gamma$. If
$t(e_1)\ne t(e_2)$ in $\Gamma$, we say that $f$ is a \emph{type-I
  Stallings fold}. If  $t(e_1)= t(e_2)$ in $\Gamma$, we say that $f$ is a \emph{type-II
  Stallings fold}.
\end{defn}

Note that if $\Gamma$ is a finite connected $\mathcal A$--graph such
that the natural projection $\Gamma\to R_\mathcal A$ is a homotopy
equivalence, and if $\Gamma'$ is obtained from $\Gamma$ by a Stallings
fold $f$, then $f$ is necessarily a type-I fold, and hence the natural projection $\Gamma'\to R_\mathcal A$ is again a homotopy
equivalence.

\begin{defn}[Maximal fold]
Let $\Gamma$ be a non-folded finite connected $\mathcal A$--graph, let $v\in VA$ be a natural vertex, let $e_1,e_2$ be two
distinct oriented edges of $\Gamma$ such that $o(e_1)=o(e_2)=v$ and such that $\mu(e_1)=\mu(e_2)=a\in \mathcal A^{\pm
  1}$.  Let $\widehat {e_1}$ and $\widehat{e_2}$ be the natural edges in $\Gamma$ that begin with $e_1$, $e_2$ accordingly. Let $z_1,z_2$ be maximal initial segments of $\widehat {e_1}$ and $\widehat{e_2}$ such that the label $\mu(z_1)$ is graphically equal, as a word over $\mathcal A^{\pm 1}$, to the label $\mu(z_2)$. Thus $z_1$ starts with $e_1$ and $z_2$ starts with $e_2$.
Let $\Gamma'$  be obtained from $\Gamma$ by a chain of Stallings folds that fold $z_1$ and $z_2$ together.
We say that $\Gamma'$ is obtained from $\Gamma$ by a \emph{maximal fold}. Being a composition of several Stallings folds, a maximal fold also comes equipped with a \emph{fold map} $f\from \Gamma\to \Gamma'$.
\end{defn}

\begin{rem}\label{rem:BT'}
Let $\Gamma$ be a connected $\mathcal A$--graph such that
$\Gamma=\Core(\Gamma)$ and such that the natural
projection $p\from \Gamma\to R_\mathcal A$ is a homotopy equivalence. Let
$\alpha\from  F_N\to \pi_1(\Gamma,v)$ be an associated marking. Let
$T\subseteq \Gamma$ be a maximal tree. 

Recall that according to Definition~\ref{defn:BT}, we have an
associated free basis $\mathcal B(\Gamma,T)$ of $F_N$. In this case $\mathcal
B(\Gamma,T)$ can be described more explicitly as follows. Choose
oriented edges $e_1,\dots, e_N\in E(\Gamma-T)$ so that
$E(\Gamma-T)=\{e_1^{\pm 1}, \dots, e_N^{\pm 1}\}$. For each
$j=1,\dots, N$ let $w_j$ be the label (i.e. a word over $\mathcal A$)
of the path $[v, o(e_j)]_{T} e_j [t(e_j), v]_{T}$. Then  $\mathcal B(\Gamma,T)=\{w_1,\dots, w_N\}$. 
\end{rem}

We need the following technical notion which is a variant of the
notion of a foldable map from the paper of Handel-Mosher~\cite{HM}.
\begin{defn}[Foldable maps]\label{defn:foldable}
Let $\Gamma$ be a finite connected $\mathcal A$--graph such that the
natural projection $p\from \Gamma\to R_\mathcal A$ is a homotopy
equivalence and such that $\Gamma=\Core(\Gamma)$.

We say that the natural projection $p\from \Gamma\to R_\mathcal A$ is
\emph{foldable} if the following conditions hold:

\begin{enumerate}
\item If $v$ is a vertex of degree 2 in $\Gamma$ and $e_1, e_2$ are
  the two distinct edges in $\Gamma$ with $o(e_1)=o(e_2)=v$ then
  $\mu(e_1)\ne \mu(e_2)$.

\item If $\degree(v)\ge 3$ in $\Gamma$ then there exist three (oriented) edges
  $e_1,e_2,e_3$ in $\Gamma$ such that $o(e_1)=o(e_2)=o(e_3)=v$ and
  such that $\mu(e_1), \mu(e_2), \mu(e_3)$ are three distinct elements
  in $\mathcal A^{\pm 1}$.

\end{enumerate}

If the natural projection $p\from \Gamma\to R_\mathcal A$ is foldable, we
will also sometimes say that the $\mathcal A$--graph $\Gamma$ is \emph{foldable}.
\end{defn}

\begin{rem}\label{rem:fold}
Let $\Gamma$ be a foldable $\mathcal A$--graph and let $\Gamma'$ be obtained
from $\Gamma$ by a maximal fold.

(1) One can check that  $\Gamma'$ is again foldable. Note, however,  that a single Stallings fold on a foldable $\mathcal A$-graph may introduce a vertex of degree three where condition (2) of Definition~\ref{defn:foldable} fails, so that the resulting graph is not foldable. Performing maximal folds instead of single Stallings folds avoids this problem.

(2) Lemma~2.5 in \cite{HM} implies that 
\[
d_{FS_N'}\left(z(\overline{\Gamma}), z(\overline{\Gamma'})\right)\le 2.
\]  
Recall that the marking $\overline\Gamma$ was defined in Remark~\ref{rem:A}.

(3) As noted above, in  \cite{HM} Handel and Mosher introduce the
notion of a ``foldable'' $F_N$-equivariant map between trees corresponding to arbitrary
minimal splittings of $F_N$ as the fundamental group of a finite graph
of groups with trivial edge groups. They also prove the existence of
such ``foldable maps'' in that setting. The general definition and construction of
foldable maps are
fairly complicated, but in the context of $\mathcal A$-graphs
corresponding to markings on $F_N$ they become much easier. In
particular, we will only need the following basic fact that follows directly
from comparing  Definition~\ref{defn:foldable} with the Handel-Mosher
definition of a foldable map:
\smallskip

Let $\Gamma$ be a finite connected $\mathcal A$--graph such that
  the natural projection $p\from \Gamma\to R_\mathcal A$ is a homotopy
  equivalence and such that $\Gamma=\Core(\Gamma)$. Suppose that $p$ is foldable in
the sense of Definition~\ref{defn:foldable} above. Then there exists a
foldable (in the sense of Handel-Mosher) map $\widetilde \Gamma\to
\widetilde  R_\mathcal A$.

\smallskip

Handel and Mosher use foldable maps as a starting point in
constructing folding paths between vertices of $FS_N'$, and we will
need the above fact in the proof of the main result in Section~\ref{sect:main}.

\end{rem}

\section{Proof of the main result}\label{sect:main}

Before giving a proof of the main result, we illustrate the relationship and the maps between $FS_N, FS_N', FF_N$ and $FB_N$ in the following diagram, provided by the referee:

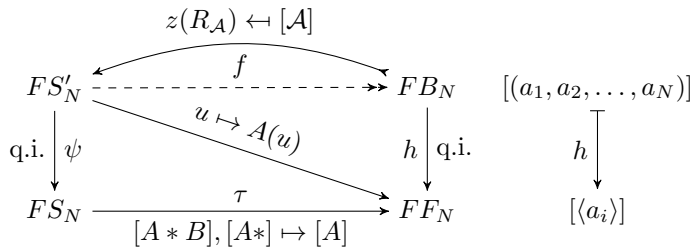
\begin{figure}

\begin{tikzpicture}
  \matrix (m) [matrix of math nodes, row sep=3em, column sep=1em]
    { FS_N'  &\hspace{3cm}& FB_N &{[(a_1,a_2,\ldots,a_N)]}  \\
      FS_N   && FF_N &{[\langle a_i\rangle]}  \\[-1cm]\\};
   \path[->>,dashed] (m-1-1) edge node[above]{$f$} (m-1-3);
   \path[->] (m-2-1) edge node[above]{$\tau$} node[below]{$[A\ast B],[A\ast]\mapsto [A]$} (m-2-3);
   \path[->] (m-1-1) edge node[right]{$\psi$} node[left]{q.i.} (m-2-1) ;
   \path[->] (m-1-3) edge node[left]{$h$} node[right]{q.i.} (m-2-3);
   \path[->] (m-1-1) edge node[above, sloped]{$u\mapsto A(u)$} (m-2-3);
   \path[>->] (m-1-3) edge [bend right=20] node[above]{$z(R_{\mathcal A})\mapsfrom [\mathcal A]$} (m-1-1);
   \path[|->] (m-1-4) edge node[left]{$h$} (m-2-4);
\end{tikzpicture}

\caption{A diagram showing the maps between the various curve complex analogs considered in this paper. Here $F_N=A\ast$ represents an HNN-extension splitting of $F_N$ with the trivial associated subgroups and with the base group $A$. The map $\psi$ is the quasi-isometry between $FS_N'$ and $FS_N$ given by the inclusion map.}\label{Fi:1}
\end{figure}

\begin{thm}\label{thm:FF_N}
Let $N\ge 3$. Then the free factor complex $FF_N$ is Gromov-hyperbolic. 
Moreover, there exists a constant $H>0$ such that for any two vertices $x,y$ 
of $FS_N$ and any geodesic $[x,y]$ in $FS_N^{(1)}$ the path $\tau([x,y])$ 
is $H$--Hausdorff close to a geodesic $[\tau(x),\tau(y)]$ in $FF_N^{(1)}$.
\end{thm}

\begin{proof}
Recall that $FS_N'$ is the barycentric subdivision of the free
splitting graph $FS_N$, so that the inclusion map $FS_N\subseteq FS_N'$ is a quasi-isometry.

Recall also that for any free basis $\mathcal A$ of
$F_N$ the rose $R_\mathcal A$ defines an $(N-1)$--simplex in $FS_N$ (via the canonical marking $F_N\to \pi_1(R_\mathcal A)$ sending the elements of $\mathcal A$ to the corresponding loop-edges of $R_\mathcal A$) and
that, as in Convention~\ref{conv:bar}, $z(R_\mathcal A)\in V(FS_N')$
is the barycenter of that simplex. Note that by definition, if
$[\mathcal A]=[\mathcal B]$ then $z(R_\mathcal A)=z(R_\mathcal B)$. Put 
\[
S=\{z(R_\mathcal A)\, |\, \mathcal A \text{ is a free basis of } F_N\}.
\]
Thus $S\subseteq V(FS_N')$ and we may think of  $S$ as a copy of $V(FB_N)$ in $V(FS_N')$.

For every $x,y\in S$  let $g_{x,y}$ be the path from $x$ to $y$ in
$FS_N'$ given by the Handel-Mosher folding line~\cite{HM}. Recall that, as proved in~\cite{HM}, $FS_N'$ is Gromov-hyperbolic and $g_{x,y}$ is a re-parameterized uniform quasigeodesic. Hence $g_{x,y}$ is uniformly Hausdorff close to any geodesic $[x,y]$ in $FS_N'$.

Consider the following map $f\from  V(FS_N')\to V(FB_N)$. For
every vertex $u$ of $FS_N'$, which may be viewed as a splitting of
$F_N$ as the fundamental group of a graph of groups with trivial
edge-groups, choose an edge $e$ of that splitting, collapse the rest of
$u$ to a single-edge splitting corresponding to $u$ and let $A(u)$ be
a vertex group of that collapsed splitting. Thus $A(u)$ is a proper free
factor of $F_N$ and hence $[A(u)]$ is a vertex of $FF_N$. Then choose
a vertex $v$ of $FB_N$ with $d(h(v), [A(u)])\le 3$ (such $v$ exists by Proposition~\ref{prop:FB}).
Put $f(u):=v$.
We can make the above choices to make sure that for every free basis $\mathcal A$ of $F_N$ we
have $f(z(R_\mathcal A))=[\mathcal A]$ (note that $d_{FF_N}(h([\mathcal A]), A(z(R_\mathcal A)))\le 2$). With the above mentioned
identification of $S$ and $V(FB_N)$ we may in fact informally think
that $f|_S=Id_S$. Moreover, if $\Gamma$ is a foldable $\mathcal
A$-graph (which therefore defines a marking $\overline \Gamma$) and
$T\subseteq \Gamma$ is a maximal tree, then we have an associated free
basis $\mathcal B(\Gamma,T)$ described in Remark~\ref{rem:BT'} above.
One can check that
$d(f(\overline\Gamma), [\mathcal B(\Gamma,T)])\le B$ for some constant
$B=B(N)>0$ independent of $\mathcal A, \Gamma, T$.

We have defined a map $f\from  V(FS_N')\to V(FB_N)$. We then extend
this map to a graph-map $f\from FS_N'\to FB_N$ by sending an
arbitrary edge
$e$ of $FS_N'$ with endpoints $u_1,u_2$ to a geodesic edge-path
$[f(u_1),f(u_2)]$ in $FB_N$. The graph-map $f \from FS_N'\to FB_N$ is 
$L$--Lipschitz for some $L \ge 0$, since $\tau$ is Lipschitz and the maps $V(FS_N')\to FF_N^{(1)}$, given by $u\mapsto [A(u)]$ and $u\mapsto \tau(u)$, are bounded distance away from each other.

We claim that all the assumptions of Proposition~\ref{prop:image} are
satisfied for the map $f\from FS_N'\to FB_N$ and the set $S$.

Condition (1) of Proposition~\ref{prop:image} holds, since by assumption $f(S)=V(FB_N)$. Also, as noted above, $f\from FS_N'\to FB_N$
is $L$--Lipschitz, and it is easy to see that $S$ is $D$--dense in $FS_N'$,
 for some $D>0$. Our task is to verify condition (3) of 
 Proposition~\ref{prop:image}.

If $x',y'\in V(FB_N)$ have $d(x',y')\le 1$ then there exist free
bases $\mathcal A, \mathcal B$ of $F_N$ such that $x'=[\mathcal B]$, $y'=[\mathcal A]$ and such that there
exists $a\in \mathcal A\cap \mathcal B$. Without loss of generality,
we may assume that $\mathcal A=\{a_1,\dots, a_N\}$, $\mathcal
B=\{b_1,\dots, b_N\}$ and that $a_1=b_1=a$. 
Put $x=z(R_\mathcal B)$ and $y=z(\mathcal A)$, so that $f(x)=x'$ and
$f(y)=y'$.

In \cite{HM} Handel-Mosher~\cite{HM}, given 
any ordered pair of vertices $x,y$ of $V(FS_N')$,
construct an edge-path $g_{x,y}$ from $x$ to $y$ in $FS_N'$, which we
will sometimes call the \emph{Handel-Mosher folding path}. 
The general definition of $g_{x,y}$ in~\cite{HM} is fairly
complicated. However, we  only need to use this definition for the case
where $x,y\in S$, in which case it becomes much simpler, and which we will now describe in greater detail for the vertices $x=z(R_\mathcal B)$ and $y=z(\mathcal A)$ defined above.

Consider an $\mathcal A$--graph $\Gamma_0$ which is a wedge of $N$
simple loops at a common base-vertex $v_0$, where the $i$-th loop is
labeled by the freely reduced word over $\mathcal A$ that is equal to
$b_i$ in $F_N$. Note that the first loop is just a loop-edge labelled
by $a_1$, since by assumption $b_1=a_1$. The natural projection $p\from \Gamma_0\to
R_\mathcal A$ is a homotopy equivalence and we also have
$\Gamma_0=\Core(\Gamma_0)$. Condition~(1) of
Definition~\ref{defn:foldable} holds for $\Gamma_0$ by
construction. However, $p\from \Gamma_0\to
R_\mathcal A$ is not necessarily foldable since Condition~(2) of
Definition~\ref{defn:foldable} may fail. This happens exactly when
there exists $\epsilon\in \{1,-1\}$ such that for all $i=2,\dots, N$
the freely reduced word over $\mathcal A$ representing $b_i$ begins
with $a_1^{\epsilon}$ and ends with $a_1^{-\epsilon}$. However,  after possibly replacing $\mathcal B$ by an
equivalent free basis of the form $a_1^m \mathcal B a_1^{-m}$, for the
graph $\Gamma_0$ defined as above the natural projection $p\from \Gamma_0\to
R_\mathcal A$ is foldable in the sense of
Definition~\ref{defn:foldable}. Note that conjugation by $a_1^m$
fixes the element $b_1=a_1$, so that even after the above modification
of $\mathcal B$ it will still be true that $\Gamma_0$ contains a
loop-edge at $v_0$ with label $a_1$.

As noted in Remark~\ref{rem:fold} above, as the initial input for
constructing $g_{x,y}$, Handel and Mosher need a ``foldable'' (in the
sense of~\cite{HM}) $F_N$--equivariant map
$\widetilde{R_\mathcal B} \to \widetilde{R_\mathcal A}$. Again, as observed in Remark~\ref{rem:fold}, such a map
exists since we have arranged for the $\mathcal A$-graph $\Gamma_0$ to
be foldable in the sense of Definition~\ref{defn:foldable}.

Note that by construction, the
marking $\overline{\Gamma_0}$ corresponding to $\Gamma_0$ is exactly
the vertex $x=z(R_\mathcal B)$ of $FS_N'$.

The remainder of the Handel-Mosher construction of $g_{x,y}$ in this
case works as follows. Since $p\from \Gamma_0\to
R_\mathcal A$ is a homotopy equivalence, there exists a finite
sequence of $\mathcal A$--graphs 
\[
\Gamma_0, \Gamma_1, \dots \Gamma_n=R_\mathcal A
\]
where for $i=1,\dots, n$ $\Gamma_{i}$ is obtained from $\Gamma_{i-1}$
  by a maximal fold. Then the associated markings
  $\overline{\Gamma_{i}}$ are vertices of $FS_N'$. As observed in Remark~\ref{rem:fold}, we have $d(\overline{\Gamma_{i-1}}, \overline{\Gamma_{i}})\le 2$ in
  $FS_N'$. Joining each consecutive pair $\overline{\Gamma_{i-1}},
  \overline{\Gamma_{i}}$ by a a geodesic path of length $\le 2$ in $FS_N'$
  produces the path $g_{x,y}$ from $x$ to $y$ in $FS_N'$.  Note that each $\Gamma_i$ has a base-vertex $v_i$ which is the image of the base-vertex $v_0$ of $\Gamma_0$ under the sequence of folds that takes $\Gamma_0$ to $\Gamma_i$.
   
A crucial feature of the above construction is that every $\Gamma_i$
will have a loop-edge (at the base-vertex $v_i$ of $\Gamma_i$) with label $a_1$. Since the map
$f\from FS_N'\to FB_N$ is $L$--Lipschitz, this implies that
$f(g_{x,y})$ has diameter bounded by some constant $M_0$ independent 
of $x,y$. 
Indeed, Since $\Gamma_i$ has a loop-edge at its base-vertex with
label $a_1$, there exists a free basis $\gamma_1,\dots, \gamma_N$ of
$\pi_1(\Gamma_i, v_i)$ (e.g. coming from a choice of a maximal tree in
$\Gamma_i$, as in Definition~\ref{defn:BT} and Remark~\ref{rem:BT'}) such that $\mu(\gamma_1)=a_1$ and such that
$\mathcal B_i=\{\mu(\gamma_1),\dots ,\mu(\gamma_N)\}$ is a free basis
of $F_N$. Since $a_1\in \mathcal B_i$, we have $d([\mathcal B_i],
[\mathcal A])\le 1$ in $FB_N$ for each $i$. Unpacking the definition of the map
$f$ we see that $d(f(\overline{\Gamma_{i-1}}),
[\mathcal B_i])\le C$ in $FB_N$ for some constant $C\ge 0$. Hence $d(f(\overline{\Gamma_{i-1}}),
[\mathcal A])\le C+1$. Recall that $g_{x,y}$ is a quasi-geodesic in a hyperbolic 
graph $FS_N'$ and hence $g_{x,y}$ is uniformly Hausdorff-close to a geodesic 
$[x,y]$. Since $f$ is  $L$--Lipschitz, it follows that $f([x,y])$ has diameter bounded by some constant $M$ independent of $x,y$. Thus condition (3) of 
Proposition~\ref{prop:image} holds.

Therefore, by  Proposition~\ref{prop:image}, the graph $FB_N$ is Gromov-hyperbolic, and, moreover, for any vertices $x,y$ of $FS_N$, the path $f([x,y])$ is uniformly Hausdorff-close to a geodesic $[f(x), f(y)]$.

Recall that in Proposition~\ref{prop:FB} we constructed an explicit quasi-isometry $h\from  FB_N\to FF_N$. Since $FB_N$ is hyperbolic, it follows that $FF_N$ is Gromov-hyperbolic as well.
Moreover,  the map $\tau\from  FS_N\to FF_N$ from the statement of Theorem~\ref{thm:A}, and the map $h\circ f\from  FS_N\to FF_N$ are bounded distance from each other. This implies that there exists a constant $H>0$ such that for any two vertices $x,y$ of $FS_N$ and any geodesic $[x,y]$ in $FS_N^{(1)}$ the path $\tau([x,y])$ is $H$--Hausdorff close to a geodesic $[\tau(x),\tau(y)]$ in $FF_N^{(1)}$.
\end{proof}

\begin{rem}\label{rem:geod}
The above proof implies a reasonably explicit description of certain reparameterized quasigeodesics in $FB_N$ between two arbitrary vertices of $FB_N$ in terms of Stallings folds. Let $\mathcal A=\{a_1,\dots, a_N\}$  and $\mathcal B=\{b_1,\dots, b_N\}$ be free bases of $F_N$. Let $\Gamma_0$ be an $\mathcal A$--graph corresponding to $\mathcal B$ constructed in a similar way to the way $\Gamma_0$ was constructed in the above proof.  That is, let $\Gamma_0$ be a wedge of $N$
simple loops at a common base-vertex $v_0$, where the $i$-th loop is
labeled by the freely reduced word over $\mathcal A$ that is equal to
$b_i$ in $F_N$. Suppose that $\Gamma_0$ is such that the natural projection $p\from  \Gamma_0\to R_\mathcal A$ is foldable in the sense of Definition~\ref{defn:foldable}. (Note that this assumption does not always hold; however, it may always be ensured after replacing $\mathcal B$ by an equivalent free basis).

Let $\Gamma_0, \Gamma_1, \dots, \Gamma_n=R_\mathcal A$ be $\mathcal A$--graphs such that for $i=1,\dots, n$ $\Gamma_{i}$ is obtained from $\Gamma_{i-1}$
  by a maximal fold. Note that each $\Gamma_i$  has a distinguished base-vertex $v_i$, which is the image of the base-vertex $v_0$ of $\Gamma_0$ under the foldings transforming $\Gamma_0$ to $\Gamma_i$
.

For each $1\le i<n$ choose a maximal subtree $T_i$ in
$\Gamma_i$. Let $\mathcal A_i=\mathcal B(\Gamma_i, T_i)$ be the
associated free basis of $F_N$ (see Remark~\ref{rem:BT'} above for its
detailed description). Put $\mathcal A_0=\mathcal B$
and $\mathcal A_n=\mathcal A$.

It is not hard to check that $d([\mathcal A_i], f(\overline \Gamma_i))\le C$ in $FB_N$ for some constant $C=C(N)>0$ independent of $\mathcal A$, $\mathcal B$.
Since, as noted in the proof of Theorem~\ref{thm:A} above, the sequence $\overline{\Gamma_0}, \dots, \overline{\Gamma_n}$ defines a (reparameterized) uniform quasigeodesic in $FS_N'$, it now follows from the proof of Theorem~\ref{thm:A} that the set $\{[\mathcal A_0], [\mathcal A_1],  \dots, [\mathcal A_n]\}$ is uniformly Hausdorff-close to a geodesic joining $[\mathcal B]$ and $[\mathcal A]$ in $FB_N$.  This fact can also be derived from a careful analysis of the Bestvina-Feighn proof~\cite{BF11} of hyperbolicity of $FF_N$.
\end{rem}

\end{document}